\documentclass[11pt]{article}
\usepackage{arxiv}
\usepackage{enumitem}
\usepackage{mathrsfs}
\usepackage{subcaption}
\usepackage{graphicx}
\usepackage{mathabx}
\usepackage{authblk}
\usepackage{comment}
\usepackage{bbold}

\mathtoolsset{showonlyrefs}

\title{Stability of the Markov-Krein Transform Under Classical Convolutions}

\author[1]{Pierre Mergny \thanks{mergny.pierre@gmail.com}}

\affil[1]{Département d’Informatique, École Normale Supérieure, PSL \& CNRS, Paris, France}

\date{\today}

\begin{document}

\maketitle

\begin{abstract}
We prove that, for every $c>0$, the image of the generalized Markov–Krein transform on probability measures with finite logarithmic moment is closed under classical additive convolution and under the Mellin convolution. This resolves the positivity conjecture for the additive high-temperature convolution \cite{MP22HTC,benaych2022matrix} a one-parameter interpolation between classical and free convolution arising in random matrix theory and representation theory in the regime $N\to\infty, N\beta_N/2\to c$. Our proof relies on a
convolution theorem for universally prestarlike functions.
\end{abstract}

\section{Introduction and Main Result}
For a Borel set $E \subseteq \mathbb{R}$, we denote by $\mathcal{P}(E)$ the set of probability measures on $E$. We define the set of probability measures with finite logarithmic moment on $\mathbb{R}$ and $\mathbb{R}_+$ by 
\begin{align} 
    \mathcal{P}_{\mathrm{log}}(\mathbb{R}) := \bigg\{ \mu \in \mathcal{P}(\mathbb{R}) : \int \log (1 + |x|) \mu(\mathrm{d}x) < \infty \bigg\},  \qquad  \mathcal{P}_{\mathrm{log}}(\mathbb{R}_+) :=  \mathcal{P}_{\mathrm{log}}(\mathbb{R}) \cap \mathcal{P}(\mathbb{R}_+). 
\end{align}
For a (Borel) map $T$ we denote by $T_\# \mu$ the pushforward of the measure $\mu$ by $T$.We denote by $\mathrm{Dir}(\alpha_1,\dots,\alpha_m)$  the law of the random variable with density $\Gamma(\alpha_1+\dots+\alpha_m)/(\prod_{j=1}^m \Gamma(\alpha_j)) \prod_{i=1}^m x_i^{\alpha_i -1}$ over the simplex $ \{ x_i \in [0,1], \,  i =1,\dots,m : \sum_{i=1}^m x_i =1 \}$ with $\alpha_1,\dots,\alpha_m>0$. For a Borel set $E$, $\mathds{1}_E$ is the indicator function of that set. We let $\mathbb{R}_+ = [0,\infty)$ and $\mathbb{C}_+ := \{ z \in \mathbb{C}: \mathfrak{I}z >0\}$. 

Our main object of study is the law of the mean of a Dirichlet process. We recall the definitions below,
using Ferguson’s construction \cite{ferguson1973bayesian} and the logarithmic-moment criterion of Lijoi–Regazzini \cite{Lijoi2009}, and refer the reader to \cite{kerov1998interlacing,FarautFourati2016} for additional results relating this object to representation theory and random matrix theory. 

\begin{definition}[Dirichlet Process, Dirichlet Mean, Markov-Krein Transform] Let $c>0$ and $\mu \in \mathcal{P}(\mathbb{R})$. A random probability measure is said to be a \emph{Dirichlet Process} with concentration $c$ and base distribution $\mu$, denoted $P \sim \mathsf{DP}(c \, \mu)$ if for every Borel partition $E_1,\dots,E_m$ of $\mathbb{R}$ we have 
\begin{align}
    (P(E_1),\dots,P(E_m)) \sim \mathsf{Dir}( c\mu(E_1), \dots, c \mu(E_m)) \, . 
\end{align}
If furthermore $\mu \in \mathcal{P}_{\mathrm{log}}(\mathbb{R})$, $\int |x| P(\mathrm{d}x) < \infty$ almost surely and the \emph{Dirichlet mean} of $P$ is defined as the random variable $Y_\mu = \int x P(\mathrm{d}x)$. We call its law the \emph{Markov-Krein Transform} (of order $c$) of the measure $\mu$ which we denote by 
\begin{align}
    \mathcal{M}_c \mu := \mathsf{Law}(Y_\mu) , \quad Y_\mu  = \int x P(\mathrm{d}x), \quad P \sim \mathsf{DP}(c \, \mu).
\end{align}
\end{definition}
The following identity relating the base measure $\mu$ to $\mathcal{M}_c \mu$, explains the terminology Markov-Krein Transform to denote the law of the Dirichlet mean. We follow here the results of  \cite[Theorem 1]{Lijoi2004}
for this relation and \cite[Theorem 2]{Lijoi2004} for injectivity.
\begin{proposition}[Markov-Krein Relation  \cite{Lijoi2004}]
\label{prop:MKR_and_property_MKT} 
For $\mu \in  \mathcal{P}_{\mathrm{log}}(\mathbb{R})$ and any $z \in \mathbb{C}_+$ we have 
\begin{align}
\label{eq:MK_relation}
    \int (z - y)^{-c} \mathcal{M}_c\mu(\mathrm{d}y) = \mathrm{exp} \bigg\{ -c \int \log ( z - x) \mu(\mathrm{d}x) \bigg\},
\end{align}
and the transform $\mathcal{M}_c$ is injective on $\mathcal{P}_{\mathrm{log}}(\mathbb{R})$.
\end{proposition}
Next, we recall the two classical convolutions. 
\begin{definition}[Classical additive and multiplicative convolutions]
\label{def:classical_convolution} For $\mu,\nu \in \mathcal{P}(\mathbb{R})$, their additive convolution $\mu \ast \nu$ and their multiplicative convolution  $\mu \circledast \nu$ are defined for any Borel set $E \subseteq \mathbb{R}$ by 
 \begin{align}
     (\mu \ast \nu)(E) = \int_{\mathbb{R}^2} \mathds{1}_E(x +y) \mu(\mathrm{d}x) \nu(\mathrm{d}y) ,
 \end{align}
 and
 \begin{align}
      (\mu \circledast \nu)(E) = \int_{\mathbb{R}^2} \mathds{1}_E(xy) \mu(\mathrm{d}x) \nu(\mathrm{d}y) .
 \end{align}
That is if $X \sim \mu, Y \sim \nu$ and $X,Y$ are independent, we have $\mu \ast \nu = \mathsf{Law}(X+Y)$ and $(\mu \circledast \nu) =   \mathsf{Law}(XY)$. The latter is known as the \emph{Mellin convolution} whenever $\mu,\nu \in \mathcal{P}(\mathbb{R}_+)$. 
\end{definition}

\noindent Our main result characterizes the stability of the Markov-Krein transform with respect to these convolutions. 

\begin{theorem}
\label{thm:main_theorem}
For any $c>0$, the following assertions hold:
\begin{enumerate}
    \item[(i)] For any $\mu, \nu \in \mathcal{P}_{\mathrm{log}}(\mathbb{R})$, there exists a unique probability measure $\mu \oplus_c \nu \in \mathcal{P}_{\mathrm{log}}(\mathbb{R})$ such that 
     \begin{align}
         \mathcal{M}_c( \mu \oplus_c \nu) = (\mathcal{M}_c \mu) \ast (\mathcal{M}_c \nu);
     \end{align}
     \item[(ii)] For any $\mu, \nu \in \mathcal{P}_{\mathrm{log}}(\mathbb{R}_+)$, there exists a unique probability measure $\mu \otimes_c \nu \in \mathcal{P}_{\mathrm{log}}(\mathbb{R}_+)$ such that 
     \begin{align}
         \mathcal{M}_c( \mu \otimes_c \nu) = (\mathcal{M}_c \mu) \circledast (\mathcal{M}_c \nu).
     \end{align} 
\end{enumerate}
\end{theorem}
This result is motivated by the frameworks of \cite{MP22HTC} in the Physics literature and \cite{benaych2022matrix}, which study the high-temperature limit of the $\beta$-sum (see \cite[Section 1]{gorin2020crystallization} for an introduction) as $N\to\infty$ and $N\beta_N/2\to c>0$. The resulting binary operation $\mu \oplus_c \nu$ interpolates between classical convolution $\mu \ast \nu$ (Def.~\ref{def:classical_convolution}) as $c\downarrow0$ and free additive convolution $\mu \boxplus\nu$ as $c\to\infty$; see \cite{mingo2017free} for an introduction to free probability. Mergny--Potters identify the limiting rank-one additive spherical integral with the moment generating function of $\mathcal M_c\mu$. while Benaych-Georges--Cuenca--Gorin develop a moment approach using multivariate Bessel functions. Both works explicitly makes the connection from this with construction to classical convolution of Markov--Krein transforms \cite[Section~5.1]{MP22HTC} \cite[Section~3.4]{benaych2022matrix}. However, the resulting distribution $\mu \oplus_c \nu$ (respectively the associated normalized moments sequence) was only conjecture to correspond to a probability measure for arbitrary inputs, see also \cite{xu2026rectangular} for a variant of this problem. Positivity was show for some specific choice of the measures $\mu$ (and/or $\nu)$ with the symmetric Bernoulli self-convolution \cite{mergny2022explicitexamplehightemperature} and, under suitable moment hypotheses, convolution with a $c$-Gaussian law through Dyson Brownian motion \cite[Theorems~IV.4 and~IV.6]{nakano2025classical}; see also \cite[Example~10.29 and Corollary~10.32]{yao2025approximating}. Our theorem resolves this positivity conjecture and extends the construction of $\mu\oplus_c\nu$ to all probability measures with finite logarithmic moment.

To the knowledge of the author, the high-temperature product counterpart has not yet been constructed in the literature. The expression of the associated rank-one multiplicative spherical transform  see \cite{mergny2021asymptoticbehaviormultiplicativecounterpart,husson2025asymptotic} suggests that it converges to the Mellin transform of $\mathcal{M}_c \mu$, as established for $c>1$ under additional hypotheses in \cite[Theorem 1.3]{zhang2025matrixharmonicanalysishigh}.  Since the same corresponding construction has been used as a building block for the high-temperature sum, this strongly suggests that the 
multiplicative operation $\mu \otimes_c \nu$ in Part (ii)  corresponds precisely to this high-temperature product. We prove the multiplicative statement independently of this identification and leave this correspondence for future work.

The remainder of this paper is organized as follows: in Sec.~\ref{sec:prestarlike}, we introduce the class of (universally) prestarlike functions and recall their main properties. The proof is given in Sec.~\ref{sec:proof} and proceeds in three steps. First, we prove the multiplicative statement for compactly supported measures on $\mathbb R_+$ by reducing it to the Hadamard convolution theorem for universally prestarlike functions. Second, we obtain the analog additive statement  through an affine contraction of the multiplicative operation. Finally, we extend both statements to measures with finite logarithmic moment using a standard weak-closure argument.

\paragraph{Acknowledgments :} The author warmly thanks Joon Hyung Lee and Jiyuan Zhang for stimulating discussions that motivated this work and Marc Potters for introducing him to this problem.

\section{Prestarlike functions}
\label{sec:prestarlike}
We recall in this section the properties of prestarlike functions that will be used later on.  Our
conventions follow essentially those of \cite[Section 1.1-1.3]{ruscheweyh2009completely}. For the classical
theory on the disk, see also \cite{ruscheweyh1982convolutions}.  Let $\mathbb{D} = \{ z \in \mathbb{C} : |z| <1 \}$ be the open disk on the complex plane and $\mathbb{\Lambda} = \mathbb{C} \setminus [1,\infty)$.  We use the Pochhammer symbol $(a)_n = a (a+1)\dots (a+n-1)$ with the standard convention $(a)_0 =1$. We recall that a holomorphic germ at $0$ means an equivalence class of functions holomorphic on neighborhoods of $0$  where two functions are identified if they coincide on some common neighborhood of $0$. 

For  $\Omega \subseteq \mathbb{C}$ a domain containing the origin, we denote by 
\begin{align}
    \mathcal{H}_1(\Omega) = \{ f \,  \mbox{holomorphic on } \Omega: f(0) = 0, f'(0) = 1  \}. 
\end{align}
Each $f \in \mathcal{H}_1(\Omega)$ has a Taylor expansion $f(z) = z \sum_{n \ge 0} f_n z^n$ with $f_0=1$  in a neighborhood of the origin and if $\Omega = \mathbb{D}$ that series is convergent on $\mathbb{D}$. 

The \emph{Hadamard product} of two power series is defined as 
\begin{align}
     \bigg(\sum_{n \ge 0} a_n z^n \bigg) \odot \bigg(\sum_{n \ge 0} b_n z^n \bigg) = \sum_{n \ge 0} a_n b_n z^n .
\end{align} 
Let $a>0$, and denote by
\begin{align}
\label{eq:binomial_series}
    b_a(z) := \frac{z}{(1-z)^{a}} = z\sum_{n \ge 0} \frac{(a)_n}{n!} z^n \in \mathcal{H}_1(\mathbb{D}) ,
\end{align}
where the second equality follows from the binomial series. For $f = z (\sum_{n \ge 0} f_n z^n)  \in \mathcal{H}_1(\mathbb{D})$,  we define the \emph{Ruscheweyh convolution operator} $\Delta^a$ (denoted $D^a$ in \cite{ruscheweyh2009completely}) on $\mathcal{H}_1(\mathbb{D})$ by
\begin{align}
    \Delta^a f=b_a \odot f = \sum_{n \geq 0} \frac{(a)_n}{n!} f_n z^{n+1}.
\end{align}
For $n \in \mathbb{N}_*$, this operator corresponds to a differential operator since 
\begin{align}
    \Delta^{n+1} f = \frac{z}{n!} \frac{\mathrm{d}^n}{\mathrm{d}z^n} \big( z^{n-1} f \big) ,\qquad \mbox{and in particular }  \Delta^{1} f =f \; \mbox{and } \Delta^{2} f  = z f^\prime.
\end{align}

\begin{definition}[Prestarlike functions on the disk and universally prestarlike functions \cite{ruscheweyh2009completely}] For $\alpha<1$, we let 
\begin{align}
    \mathcal{S}_\alpha := \bigg\{ f \in \mathcal{H}_1(\mathbb{D}): \frac{f(z)}{z} \neq 0 \mbox{ on } \mathbb{D}, \;  \mathfrak{R} \frac{z f'(z)}{f(z)} > \alpha  \mbox{ on } \mathbb{D} \bigg\} .  
\end{align}
Note that the functions $\frac{f}{z} , \mathfrak{R} \frac{z f'}{f}$ are assigned their removal value $1$ at the origin. We define the \emph{class of prestarlike functions of order $\alpha$} has 
\begin{align}
    \mathcal{R}_\alpha := \bigg\{ f \in \mathcal{H}_1(\mathbb{D}) : \Delta^{2(1-\alpha)} f \in \mathcal{S}_\alpha \bigg\}. 
\end{align}
For $\gamma \in \mathbb{C}\setminus \{0\}$ and $t \in [0,1]$ let
\begin{align}
    \omega_{\gamma,t}(z) = \frac{\gamma z}{1 -tz }, \quad z \in \mathbb{D}.
\end{align}
For $| \gamma + t| <1$, $   \omega_{\gamma,t} (\mathbb{D}) \subseteq \mathbb{\Lambda}$ and the \emph{class of universally prestarlike functions of order $\alpha$} is defined as
\begin{align}
    \mathcal{R}^{(\mathrm{u})}_{\alpha} :=\{ f \in \mathcal{H}_1(\mathbb{\Lambda}) : \gamma^{-1} f \circ \omega_{\gamma,t} \in \mathcal{R}_\alpha \mbox{ for every } t \in [0,1], \gamma \neq 0, |\gamma +t| \le 1\}. 
\end{align}
\end{definition}
We also introduce the \emph{Stietljes class} of \cite[Section 1.1]{ruscheweyh2009completely}
\begin{align}
\label{eq:def_Stieltjes_class}
    \mathcal{T} 
    &:=
    \bigg \{ \mathbb{\Lambda} \ni z \mapsto \int_{[0,1]} \frac{\rho(\mathrm{d}x)}{1 -zx} : \rho \in \mathcal{P}([0,1]) \bigg\}.
\end{align}
Every $f \in \mathcal{T}$ has $f(0)=1$ and is also holomorphic on $\mathbb{\Lambda}$ , see \cite{ruscheweyh2009completely}. The two results below are taken from \cite[Theorem 1.1]{ruscheweyh2009completely} for (i) and the convolution closure theorem of \cite{Ruscheweyh2008} for (ii).
\begin{theorem}[\cite{ruscheweyh2009completely,Ruscheweyh2008}]
\label{thm:prestartlike_stieltjes} Let $\alpha<1$, the following hold: 
\begin{enumerate}
    \item[(i)] for $f \in \mathcal{H}_1(\mathbb{\Lambda})$, we have  $ f \in \mathcal{R}^{(\mathsf{u})}_\alpha \Leftrightarrow \frac{\Delta^{3-2\alpha} f}{\Delta^{2(1-\alpha)}f} \in \mathcal{T}$;
    \item[(ii)] if $f,g \in \mathcal{R}^{(\mathsf{u})}_\alpha$ then $f \odot g$ has a holomorphic continuation to  $\mathbb{\Lambda}$ and  $f \odot g \in \mathcal{R}^{(\mathsf{u})}_\alpha$. 
\end{enumerate}
\end{theorem}
The quotient in Part-(i) is initially understood as a germ at the origin. Since it belongs to $\mathcal{T}$, it admits a holomorphic continuation to $\mathbb{\Lambda}$. 

To connect with the previous Markov-Krein formalism with parameter $c>0$, note that we will be using this theorem with $\alpha = 1-c/2 <1$. Since $(c+1)_n/(c)_n = (1+n/c) $, and $z (\Delta^{c}f)' - \Delta^{c}f = \sum_{n\ge 0} \frac{ n (c)_n}{n!}f_n z^{n+1}$,  we have $\Delta^{c+1} f= \Delta^{c} f + \frac{1}{c} \big(  z (\Delta^{c}f)' - \Delta^{c}f \big)$, and in particular the quotient in Part-(i) writes
\begin{align}
\label{eq:quotient_simplication}
    \frac{\Delta^{c+1} f}{\Delta^{c}f} = 1 + \frac{1}{c} \bigg(  \frac{z (\Delta^{c}f)'(z) }{\Delta^{c}f(z)} -1 \bigg) \, .
\end{align}

\section{Proof of the Main Theorem}
\label{sec:proof}
In the following, for $E \subseteq \mathbb{R}$, we denote by $\mathcal{P}_\mathrm{c}(E)$ the set of compactly supported probability measures on $E$. We write $\mathrm{Supp} \mu$ for the support of $\mu$. $\Rightarrow$ denotes weak convergence while $\xrightarrow[]{\mathrm{a.s.}}$ denotes almost-sure convergence. $\mathrm{i} = \sqrt{-1}$.

\subsection{Additional Properties of the Markov-Krein Transform}
We recall and give additional properties of the Markov-Krein Transform that will be used later on. 
\begin{lemma}[Support and continuity on a common compact interval] Let $c>0$.
\label{lem:support_MK}
\begin{enumerate}
    \item[(i)] If $\mu \in \mathcal{P}_{\mathrm{c}}(\mathbb{R})$ then $\mu$ and $\mathcal{M}_c \mu$ have the same smallest closed interval containing their respective supports. 
    \item[(ii)] If $\mu_n \Rightarrow \mu$ with $\mu_n \in \mathcal{P}_{\mathrm{c}}(\mathbb{R})$ supported in the same compact interval then $\mathcal{M}_c \mu_n \Rightarrow \mathcal{M}_c \mu$.
\end{enumerate}
\end{lemma} 

\begin{proof}
Part-(i) is given in \cite[Section 3.2]{Lijoi2009}. Part (ii) follows from the continuity of the Dirichlet process \cite[Theorem 3.10]{DelloSchiavo2019}. 
\end{proof}

\begin{lemma}[The Markov-Krein Transform commutes with affine shifts]
\label{lem:affine_shift}
Fix $c>0$ and for  $a ,b \in \mathbb{R}$, let $T_{a,b} : x \mapsto ax  +b$ then for any $\mu \in\mathcal{P}_{\mathrm{log}}(\mathbb{R})$ we have $\mathcal{M}_c \big( {(T_{a,b})}_\# \mu \big) = {(T_{a,b})}_\# \mathcal{M}_c  \mu$. 
\end{lemma}
\begin{proof}
For $P \sim \mathsf{DP}( c \, \mu)$ we have ${(T_{a,b})}_\# P \sim \mathsf{DP}( c \,  ({(T_{a,b})}_\# \mu) $ which by pushing forward the map, has a Dirichlet mean $a \int x P(\mathrm{d}x) + b$ which gives the desired result. We note that this relation can be also obtained  from the Markov-Krein relation \eqref{eq:MK_relation}. 
\end{proof}

\begin{lemma}
\label{lem:weak_closure}
 Let $c>0$ and a sequence  $(\rho_n)_n$ with $\rho_n \in\mathcal{P}_{\mathrm{log}}(\mathbb{R})$. If $\eta_n = \mathcal{M}_c \rho_n \Rightarrow \eta \in \mathcal{P}(\mathbb{R})$, then there exists a unique $\rho \in \mathcal{P}_{\mathrm{log}}(\mathbb{R})$ such that $\eta = \mathcal{M}_c \rho$ and $\rho_n \Rightarrow \rho$.   
\end{lemma}
\begin{proof}
 Let $c>0$ and $\tau \in \mathcal{P}(\mathbb{R})$, we introduce the standard and generalized Stieltjes transform of order $c>0$ of $\tau$ as
\begin{align}
\label{eq:definition_stietljes}
    g_{\tau}(z) := \int_{\mathbb{R}} \frac{1}{z-x} \tau(\mathrm{d}x) ,  
    \qquad 
    G^{(c)}_\tau(z) := \int_{\mathbb{R}} (z-x)^{-c} \tau(\mathrm{d}x) , \qquad z \in \mathbb{C}_+. 
\end{align}
Weak convergence of the sequence $\eta_n$ gives that $g_{n} := g_{\eta_n} \to g_{\eta}$ and $  G^{(c)}_n :=   G^{(c)}_{\eta_n} \to   G^{(c)}_\eta$ locally uniformly on $\mathbb{C}_+$ and for every $\tau \in \mathcal{P}(\mathbb{R})$ we have 
\begin{align}
\label{eq:limit_gen_stietjes_inf}
    (\mathrm{i}y)^c G^{(c)}_\tau(\mathrm{i}y)  \to 1, \qquad ( \mathrm{i} y)^{c+1} {(G^{(c)}_\tau)}^\prime ( \mathrm{i} y)  \to -c, \qquad \mbox{ as } y\to\infty,
\end{align}
hence $G^{(c)}_\tau$ is not identically zero. Form the Markov-Krein relation of Prop.~\ref{prop:MKR_and_property_MKT} and Hurwitz theorem \cite[Chapter VIII,Section 3]{gamelin2001complex} $G^{(c)}_\eta(z) \neq 0$ for all $z \in \mathbb{C}_+$ and differentiating the Markov-Krein relation \eqref{eq:MK_relation} and taking the limit $n \to \infty$ gives
\begin{align}
\label{eq:convergence_D_MKR}
    g_{\rho_n} = -\frac{1}{c} \frac{ (G^{(c)}_{\eta_n})^\prime  }{G^{(c)}_{\eta_n}} 
    &\to
    g = 
    -\frac{1}{c} \frac{ (G^{(c)}_{\eta})^\prime  }{G^{(c)}_{\eta}},
\end{align}
locally uniformly on $\mathbb{C}_+$ and Eq.~\eqref{eq:limit_gen_stietjes_inf} gives the asymptotic $\mathrm{i}y g(\mathrm{i}y) \to 1$ as $y \to \infty$. By Stieltjes continuity theorem \cite{geronimo2003necessary}, there is a probability measure $\rho$ such that $g = g_\rho$ and $\rho_n \Rightarrow \rho$. 
We check that $\rho$ has finite logarithmic moment. Form the righthand side of Eq.~\eqref{eq:convergence_D_MKR}  we have for $r>0$:
\begin{align}
    \frac{\mathrm{d}}{\mathrm{d}r}\log  |G^{(c)}_\eta (\mathrm{i}r)| = - c \int_{\mathbb{R}} \frac{r}{x^2 +r^2} \rho(\mathrm{d}x), 
\end{align}
which by integrating from $r=1$ to $r=R>1$ and adding $\log R^c =c \log R$ to both sides gives
\begin{align}
    \log  R^c |G^{(c)}_\eta (\mathrm{i}R)| -\log  |G^{(c)}_\eta (\mathrm{i})| = \frac{c}{2} \int_{\mathbb{R}} \log \bigg( \frac{1+x^2}{1+\frac{x^2}{R^2}} \bigg)\rho(\mathrm{d}x)
    \le \frac{c}{2} \int_{\mathbb{R}} \log (1+x^2)\rho(\mathrm{d}x).
\end{align}
Now by Eq.~\eqref{eq:limit_gen_stietjes_inf}, $\log  R^c |G^{(c)}_\eta (\mathrm{i}R)|\to 0$ as $R\to \infty$ and since  $0 \le \log \bigg( \frac{1+x^2}{1+\frac{x^2}{R^2}} \bigg) \le \log (1+x^2)$ with equality as $R\to\infty$, we obtain by monotone convergence
\begin{align}
    \int_{\mathbb{R}} \log (1+x^2)\rho(\mathrm{d}x) \le -\frac{2}{c} \log  |G^{(c)}_\eta (\mathrm{i})| < \infty . 
\end{align}
Since $\log(1+|x|) \le \log(1+x^2) +\frac{1}{2} \log 2$, we have $\rho \in \mathcal{P}_\mathrm{log}(\mathbb{R})$.  The logarithmic-derivative characterization \cite[Lemma IV.7]{nakano2025classical} gives $\eta = \mathcal{M}_c \rho$. Uniqueness follows from injectivity of $\mathcal{M}_c$. 
\end{proof}

\subsection{Proof of Part-(ii) of Thm.~\ref{thm:main_theorem} for compactly supported measures}

\begin{proposition}[Stability under multiplicative convolution for compactly supported measures]
\label{prop:Multiplicative_conv_compact_support}
For $c>0$ and $\mu, \nu \in \mathcal{P}_{\mathrm{c}}(\mathbb{R}_+)$  there exists a unique probability measure $\rho \in \mathcal{P}_{\mathrm{c}}(\mathbb{R}_+)$ such that $\mathcal{M}_c \rho = (\mathcal{M}_c \mu) \circledast (\mathcal{M}_c \nu)$. Furthermore if $\mathrm{Supp} \mu \subseteq [0,a]$ and $\mathrm{Supp} \nu \subseteq [0,b]$, then $\mathrm{Supp} \rho \subseteq [0, ab]$.  
\end{proposition}
\begin{proof}
By Lemma~\ref{lem:affine_shift} it suffices to prove the result for $\mu,\nu \in \mathcal{P}([0,1])$ since for $a,b>0$ we can always apply the linear push forward maps $x\mapsto x/a$ and  $x\mapsto x/b$  and then dilate the output distribution by $x \mapsto ab \, x$. If $a=0$ and/or $b=0$, one input is zero and so is the output.  Let $Y_\mu$, $Y_\nu$ (independent) be their associated Dirichlet means of order $c$, taking values in $[0,1]$ by Lemma~\ref{lem:support_MK}. We define 
\begin{align}
\label{eq:expression_gen_stieltjes}
    G^{(c)}_\mu(z)  := \mathbb{E}_{Y_\mu} \Big[ ( 1 - zY_\mu)^{-c} \Big] = \mathrm{exp} \Big\{- c \int_{[0,1]} \log (1 - zx) \mu(\mathrm{d}x) \Big\} \, , \qquad z \in \mathbb{\Lambda},
\end{align}
where the second equality follows from the Markov-Krein relation \eqref{eq:MK_relation}. Similarly let
\begin{align}
    H_\mu(z) := z  \mathbb{E}_{Y_\mu} \Big[ ( 1 - zY_\mu)^{-1} \Big], \qquad   z \in \mathbb{\Lambda}. 
\end{align}
 These functions are holomorphic on $\mathbb{\Lambda}$ with ${H}^{\prime}_{\mu}(z) = \mathbb{E}_{Y_\mu} \Big[ ( 1 - zY_\mu)^{-2} \Big]$, thus $H_\mu(0)=0$ and ${H}^{\prime}_{\mu}(0)=1$ that is $H_\mu \in \mathcal{H}_1(\mathbb{\Lambda})$. The binomial series \eqref{eq:binomial_series} gives that these two functions are linked by 
 \begin{align}
     \Delta^{c} H_\mu =  z G^{(c)}_\mu(z), 
 \end{align}
 such that, thanks to Eq.~\eqref{eq:quotient_simplication}, we can express the quotient as 
 \begin{align}
     \frac{\Delta^{c+1} H_\mu}{\Delta^{c} H_\mu} 
     &= 
      1 + \frac{1}{c}  \frac{ z (G^{(c)}_\mu)'(z)}{ G^{(c)}_\mu(z)} ,
 \end{align}
 and by differentiating the right-hand-side of Eq.~\eqref{eq:expression_gen_stieltjes} we obtain 
 \begin{align}
      \frac{\Delta^{c+1} H_\mu}{\Delta^{c} H_\mu} 
     &= \int_{[0,1]} \frac{1}{1-zx} \mu(\mathrm{d}x) \in \mathcal{T}
 \end{align}
by Eq.~\eqref{eq:def_Stieltjes_class}. Thus, Part-(i) of Thm.~\ref{thm:prestartlike_stieltjes} implies that for all $\mu \in \mathcal{P}([0,1])$, $H_\mu \in \mathcal{R}^{(\mathrm{u})}_{1-c/2}$. 

    Next, let  $Y = Y_\mu Y_\nu$ taking values in $[0,1]$ and define 
    \begin{align}
    \label{eq:def_Htilde}
        {\tilde{H}}_{Y}(z) 
        &=
        z \mathbb{E}_{Y}[ (1 -z Y)^{-1} ] \qquad z \in \mathbb{\Lambda}, 
    \end{align}
which is again holomorphic on $\mathbb{\Lambda}$. By independence of $Y_\mu, Y_\nu$ we have 
\begin{align}
     {\tilde{H}}_{Y}(z) 
        &=
         z \sum_{n\ge 0} \mathbb{E}[Y^n] z^n = z \sum_{n\ge 0} \mathbb{E}[Y_\mu^n] \mathbb{E}[Y_\nu^n] z^n =  (H_\mu \odot H_\nu)(z) &\qquad z \in \mathbb{D}. 
\end{align}
The series equality hold on $\mathbb{D}$ and is continued on $\mathbb{\Lambda}$ by Eq.~\eqref{eq:def_Htilde}.  
Hence by Part-(ii) of Thm.~\ref{thm:prestartlike_stieltjes}, $ \tilde{H}_{Y} \in \mathcal{R}^{(\mathrm{u})}_{1-c/2}$. We can now apply the direct implication of Part-(i) of Thm.~\ref{thm:prestartlike_stieltjes} to $  {\tilde{H}}_{Y}$: there exists $\rho \in \mathcal{P}([0,1])$ such that 
\begin{align}
      \frac{\Delta^{c+1} {\tilde{H}}_{Y}}{\Delta^{c}  {\tilde{H}}_{Y}} 
     &= \int_{[0,1]} \frac{1}{1-zx} \rho(\mathrm{d}x) .
 \end{align}
Define  on $\mathbb{\Lambda}$, $\tilde{G}^{(c)}_Y(z) := \mathbb{E}[ (1-zY)^{-c}]$. The same previous computation gives
 \begin{align}
     \frac{(\tilde{G}^{(c)}_Y)^{\prime}(z)}{\tilde{G}^{(c)}_Y(z)} = c \int_{[0,1]} \frac{x}{1-zx} \rho(\mathrm{d}x), 
 \end{align}
 since $\tilde{G}^{(c)}_Y(0)=1$ we get by integration and the use of the Markov-Krein relation  Eq.~\eqref{eq:MK_relation} (first in a neighborhood of the origin and then by continuation on $\mathbb{\Lambda})$: 
 \begin{align}
     \tilde{G}^{(c)}_Y(z) 
     &= \mathrm{exp} \bigg\{ -c \int_{[0,1]} \log (1 - zx) \rho(\mathrm{d}x) \bigg\}= \mathbb{E}[(1 - zY_\rho)^{-c}] =: G_\rho^{(c)}(z), \quad z \in \mathbb{\Lambda}. 
 \end{align}
Equality of the Taylor coefficients gives
\begin{align}
\mathbb{E}[Y^n] = \mathbb{E}[Y_\rho^n], \quad n \in \mathbb{N},
\end{align}
and since the two distributions are compactly supported, they are determined by their moments, that is $\mathsf{Law}(Y) =  \mathcal{M}_c \rho$ and injectivity of $\mathcal{M}_c$ gives uniqueness and  concludes the proof. 
\end{proof}

\subsection{Proof of Part-(i) of Thm.~\ref{thm:main_theorem} for compactly supported measures.}

\begin{proposition}[Stability under additive convolution for compactly supported measures]
\label{prop:additive_conv_compact_support}For $c>0$ and $\mu, \nu \in \mathcal{P}_{\mathrm{c}}(\mathbb{R})$  there exists a unique probability measure $\varrho \in \mathcal{P}_{\mathrm{c}}(\mathbb{R})$ such that $\mathcal{M}_c \varrho = (\mathcal{M}_c \mu) \ast (\mathcal{M}_c \nu)$. Moreover if $\mathrm{Supp} \mu \subseteq [a_-,a_+]$ and $\mathrm{Supp} \nu \subseteq [b_-,b_+]$, then $\mathrm{Supp} \varrho \subseteq [a_- + b_-, a_+ + b_+]$.  
\end{proposition}

\begin{proof}
    Let $\epsilon>0$, and consider the affine maps $A_\epsilon : x \mapsto 1 + \epsilon x$ and $A^{-1}_\epsilon : x \mapsto (x-1)/\epsilon$. Let  $\epsilon_0 = 1/(1+\max(0,-a_-,-b_-))$, then for $\epsilon \in (0,\epsilon_0)$ we have ${(A_\epsilon)}_\# \mu ,{(A_\epsilon)}_\# \nu \in \mathcal{P}_{\mathrm{c}}(\mathbb{R}_+)$. From Prop.~\ref{prop:Multiplicative_conv_compact_support}, there exists $\rho_\epsilon \in \mathcal{P}_{\mathrm{c}}(\mathbb{R}_+)$ such that 
    \begin{align}
    \label{eq:rho_epsilon}
        \mathcal{M}_c(\rho_\epsilon)
        &=
       \mathcal{M}_c( {(A_\epsilon)}_\# \mu) \circledast  \mathcal{M}_c( {(A_\epsilon)}_\# \nu).
    \end{align}
    For $\epsilon \in (0,\epsilon_0)$, the positive random variable $(1+\epsilon Y_\mu)(1+\epsilon Y_\nu)$ takes values in 
    \begin{align}
        I_{\epsilon} := [(1 + \epsilon a_-)( 1+ \epsilon b_-) , (1 + \epsilon a_+)( 1+ \epsilon b_+)  ] \subseteq \mathbb{R}_+,
    \end{align}
    and by Part-(i) of Lemma \ref{lem:support_MK}, we also have $\mathrm{Supp} \rho_\epsilon \subseteq  I_{\epsilon} $. 
    
    Next define $\varrho_\epsilon := (A^{-1}_\epsilon)_\# \rho_\epsilon$. By applying $A^{-1}_\epsilon$ to Eq.~\eqref{eq:rho_epsilon} and since the Markov Krein transform commutes with affine maps (Lemma~\ref{lem:affine_shift}) we have
    \begin{align}
    \label{eq:MKR_epsilon}
         \mathcal{M}_c \varrho_\epsilon  
        &=  \mathsf{Law}\Big( \epsilon^{-1}(  (1+\epsilon Y_\mu)(1+\epsilon Y_\nu) -1)\Big) 
        = \mathsf{Law}( Y_\mu + Y_\nu + \epsilon Y_\mu Y_\nu). 
    \end{align}
    Furthermore, 
    \begin{align}
    \label{eq:bound_support_eps}
        \mathrm{Supp} \varrho_\epsilon \subseteq J_\epsilon := A^{-1}_\epsilon I_\epsilon = [ a_- + b_- + \epsilon a_-b_- , a_+ + b_+ + \epsilon a_+ b_+]. 
    \end{align}
    For $\epsilon \in (0,\epsilon_0)$, there exists a $K>0$ such that $J_\epsilon \subseteq [-K,K]$ and therefore the family $(\varrho_{\epsilon})_\epsilon$ is tight. By Prokhorov's theorem \cite[Section 5]{billingsley1999convergence}, starting with a sequence $\epsilon_n \downarrow 0$, there exists a subsequence, still denoted $(\epsilon_n)_n$ such that $\varrho_{\epsilon_n} \Rightarrow \varrho \in \mathcal{P}([-K,K])$. By Part (ii) of Lemma~\ref{lem:support_MK} and  the limit $\epsilon \downarrow 0$  of Eq.~\eqref{eq:MKR_epsilon}, we have  $\mathcal{M}_c \varrho_{\epsilon_n} \Rightarrow \mathsf{Law}( Y_\mu + Y_\nu) = \mathcal{M}_c \mu \ast \mathcal{M}_c \nu$ with the associated bound for the support following from Eq.~\eqref{eq:bound_support_eps}. Injectivity of $\mathcal{M}_c$ gives again uniqueness and implies that any subsequential limit of $(\varrho_\epsilon)_\epsilon$ converges to $\varrho$ which concludes the proof. 
\end{proof}

\subsection{Proof of Thm.~\ref{thm:main_theorem} for unbounded measures with finite logarithmic moment}
For $R>0$, let $T_R : x \mapsto \max(\min(x,R),-R)$. Let $\mu,\nu \in \mathcal{P}_{\mathrm{log}}(\mathbb{R})$ and denote by $P_\mu \sim \mathsf{DP}( c \, \mu) $ and $P_\nu \sim \mathsf{DP}( c \, \nu)$, independent. Consider the clipped measures $\mu_R = (T_R)_\# \mu$ and $\nu_R = (T_R)_\# \nu$,  pushing forward the map $T_R$ on the corresponding Dirichlet means gives
\begin{align}
    Y_{\mu,R} := Y_{\mu_R} = \int T_R(x) P_\mu(\mathrm{d}x) \qquad \mbox{and} \qquad     Y_{\nu,R} := Y_{\nu_R} = \int T_R(x) P_\nu(\mathrm{d}x).
\end{align}
Since $|T_R(x)| \le | x| $ with $T_R(x) \underset{R\to\infty}{\to} x$, path-wise dominated convergence gives $ Y_{\mu,R}  \xrightarrow[R\to \infty]{\mathrm{a.s.}} Y_\mu$ and $ Y_{\nu,R}  \xrightarrow[R\to \infty]{\mathrm{a.s.}} Y_\nu$. 

 Prop.~\ref{prop:additive_conv_compact_support} gives the existence of $\varrho_R \in \mathcal{P}_{\mathrm{c}}(\mathbb{R})$ such that 
 \begin{align}
     \mathcal{M}_c \varrho_R  &=  \mathsf{Law}( Y_{\mu,R} + Y_{\nu,R}) \Rightarrow \mathsf{Law}( Y_{\mu} + Y_{\nu}) \quad \mbox{as } R \to \infty,
 \end{align}
 and applying Lemma~\ref{lem:weak_closure} gives the existence of $\varrho \in \mathcal{P}_{\mathrm{log}}(\mathbb{R})$ such that $\varrho_R \Rightarrow \varrho$ and $\mathcal{M}_c \varrho = \mathsf{Law}( Y_{\mu} + Y_{\nu})$ which gives the desired result for the additive convolution. The proof for the multiplicative case follows identically by replacing $\mathbb{R}$ by $\mathbb{R}_+$, $Y_{\mu,R} + Y_{\nu,R}$  by $Y_{\mu,R} Y_{\nu,R}$ and applying Prop.~\ref{prop:Multiplicative_conv_compact_support} in place of  Prop.~\ref{prop:additive_conv_compact_support}. Injectivity of $\mathcal{M}_c$ gives uniqueness in both cases and concludes the proof of the main theorem.

\bibliographystyle{alpha}
\bibliography{references}

\end{document}